\documentclass[a4paper,12pt]{article}

\usepackage[english]{babel}
\usepackage{graphicx}
\usepackage{amssymb}
\usepackage{amsthm}
\usepackage{setspace}

\usepackage{amsmath,amsthm,amscd,amsfonts,mathrsfs,amssymb}
\usepackage{graphicx,enumerate}

\usepackage{mathpazo}

\usepackage[all]{xy}
\usepackage{fullpage} 
\usepackage{color}
\usepackage{makeidx}
\usepackage{alltt}
\usepackage{setspace}

\newtheorem{theorem}{Theorem}[section]
\newtheorem{lemma}[theorem]{Lemma}

\theoremstyle{definition}
\newtheorem{remark}[theorem]{Remark}
\theoremstyle{definition}

\newcommand{\dd}{\;{\rm{d}}}
\newcommand{\ssl}[2]{\textrm{SL}(#1,\mathbb{#2})}
\newcommand{\sym}{\textrm{Sym}}
\newcommand{\summ}{\sum_{n=1}^\infty}

\newcommand{\real}{\Re}
\newcommand{\lambdaf}{\lambda_\phi}
\newcommand{\twopii}{\frac{1}{2\pi i}}
\newcommand{\gammao}{\Gamma_{o}}
\newcommand{\adjoint}{{\rm Ad}\;}
\newcommand{\gjf}{A_\phi}
\newcommand{\chif}{\chi_\phi}
\newcommand{\andd}{\quad\text{ and }\quad}
\newcommand{\bra}[1]{\left( #1 \right)}
\newcommand{\sigmaa}{3/4}
\newcommand{\sbb}{\mathfrak{S}}
\newcommand{\sumb}[2]{\underset{#1}{\overset{#2}{\;\sum{}^\flat}}}
\newcommand{\gjff}{A_\phi^{[4]}}
\newcommand{\gjft}{A_\phi^{[3]}}
\newcommand{\raman}{\#\{p\leq X,|\lambda _{\phi}(p)| \leq 2\}}

\begin{document}

\title{\scshape On the Hecke Eigenvalues of Maass Forms}

\author{Wenzhi Luo\footnote{Research of W. Luo is
partially supported by NSF grant DMS-1160647.}$\;$ and Fan Zhou}
\maketitle

\begin{abstract}
Let $\phi$ denote a primitive Hecke-Maass cusp form for $\Gamma_o(N)$ with the Laplacian eigenvalue $\lambda_\phi=1/4+t_{\phi}^2$. In this work we show that there exists a prime $p$ such that $p\nmid N$, $|\alpha_{p}|=|\beta_{p}| = 1$, and $p\ll(N(1+|t_{\phi}|))^c$, where $\alpha _{p},\;\beta _{p}$ are the Satake parameters of $\phi$ at $p$, and $c$ is an absolute constant with $0<c<1$. In fact, $c$ can be taken as $0.27332$. In addition, we prove that the natural density of such primes $p$ ($p\nmid N$ and $|\alpha_{p}|=|\beta_{p}| = 1$) is at least 34/35.
\end{abstract}


\section{Introduction}
The celebrated Ramanujan-Petersson conjecture for an elliptic
cuspidal Hecke eigenform $f$ of weight $k \geq 2$ and level $N$ asserts that
for any prime $p \nmid N$,
$$|\lambda_f (p)| \leq 2 p^{\frac{k-1}{2}},$$
where $\lambda_f(p)$ denotes the $p$-th Hecke eigenvalue of $f$. This 
conjecture has been solved affirmatively by Deligne in \cite{De1} and \cite{De2} as a consequence of his proof of the Weil conjectures.

 Now let $\phi$ denote a primitive Hecke-Maass cusp form  for
$\gammao(N)$ and Dirichlet character $\chif$ with the Laplacian eigenvalue 
$\lambdaf = 1/4 + t_{\phi}^{2}$.
Denote the $n$-th Hecke eigenvalue of $\phi$ by $\lambdaf(n)$ for $n\in \mathbb{N}$.
The generalized Ramanujan-Petersson conjecture predicts that
for $p \nmid N$,
$$ |\lambdaf (p)| \leq 2,$$
which is equivalent to (see the Lemma \ref{lemma1} below)
$|\alpha _{p}| = |\beta _{p}| = 1$,
where $\{\alpha _{p},\;\beta _{p}\}$ are the Satake parameters of $\phi$ 
at $p$,  i.e. the local component of $\phi$ at $p$ is tempered. 
This is an outstanding unsolved problem in number theory, which 
would follow
from  the Langlands functoriality conjectures. 
Currently the record of individual bounds towards this conjecture is due to Kim-Sarnak \cite{KS} 
\begin{equation*}
|\lambdaf (p)| \leq p^{\frac{7}{64}}+p^{-\frac{7}{64}},
\end{equation*}
a culmination of a chain of advances in the theory of 
automorphic forms and analytic number theory.

In a different direction, it is proved by Ramakrishnan in \cite{rama} that for a Maass form
$\phi$ as above, this 
conjecture (i.e., $|\alpha_p|=|\beta_p|=1$)
 is true for (unramified) primes
with the lower \textit{Dirichlet density} at least 9/10.
This lower Dirichlet density is later improved to 34/35 in \cite{KSh}. 
 For simplicity the  primes at which the Ramanujan conjecture holds are referred as the \textit{Ramanujan primes} of $\phi$, so 
the Ramanujan conjecture is equivalent to the statement
that all (unramified) primes are Ramanujan primes of $\phi$.
Note that the method in \cite{rama} (and \cite{KSh}) is ineffective, and does not provide any quantitative
bound, for example, for the occurrence of the least Ramanujan prime for a given Maass form $\phi$.

The main purpose of this paper is to show that the least Ramanujan prime
of $\phi$ is bounded by $\left( N( 1+|t_{\phi}|)\right)^c$ for some constant $c > 0$, and in 
fact we can prove a 
'subconvexity' bound with $c < 1$ (see Section \ref{section2} and Section \ref{section3} below). 
Indeed, such a result would be a direct consequence of a still open
subconvexity bound for 
automorphic $L$-functions on $GL(3)$ in the eigenvalue aspect. Furthermore, the  Lindel\"of hypothesis (a consequence of the Riemann Hypothesis) for the adjoint $L$-function of $\phi$ (see (\ref{adjoint}) below) would imply that the exponent $c>0$ could be taken arbitrarily small.

 Our approach is based upon the following simple yet crucial observation
that if an 
unramified prime $p$ is {\it not} a Ramanujan prime of $\phi$, then 
(see Lemma \ref{lemma1} below)
$\lambdaf(p^{2i}) \overline{\chif}(p^{i}) > 2i+1$ 
for all $i \geq 1$, 
where $\chif$ is the central character of $\phi$.
 Thus the following adjoint (square) $L$-function associated to $\phi$ 
 comes into play (see \cite{GJ}),
\begin{equation}\label{adjoint}
 L(s,\adjoint\phi) \; = 
\frac{L(s, \phi\times \overline{\phi})}{\zeta (s)} 
\; = \zeta ^{(N)} (2s) \; \summ \overline{\chif}(n)
\lambdaf(n^{2}) n^{-s}, 
\end{equation}
where $\zeta ^{(N)}(s)$, as usual, stands for the 
partial zeta function with local factors at $p | N$ removed from 
$\zeta (s)$. 
 Then naturally
we can relate our goal
of bounding the least unramified Ramanujan prime for Maass form $\phi$
to the sieving idea in the work \cite{IKS}
(as well as its further 
refinements in \cite{K} and \cite{mato}), 
which study the first negative Hecke eigenvalue
for a holomorphic Hecke eigenform based on the Deligne's resolution of
Ramanujan-Petersson conjecture in the case of elliptic modular forms.

It turns out that the sieving idea in \cite{IKS} (also in \cite{K} and \cite{mato})
works well in the current quite different setting, even though the 
Deligne-type bound is not available yet for Maass form $\phi$.

We present two proofs with different exponents $c$. The first proof (Section \ref{section2}) illustrates our basic ideas via the simple case of level 1. The second proof obtains significantly better (smaller) exponent $c$.
In Section \ref{four}, we refine the density results in \cite{rama} and \cite{KSh} from the Dirichlet density to the natural density.

We end the Introduction by stating the following Lemma \ref{lemma1}, which will
be used in the proofs of the following sections,  and a part of it is also an ingredient in \cite{rama}.

\begin{lemma}\label{lemma1}
Let $\{\alpha _{p}, \beta _{p}\}$ denote the Satake parameters at $p \nmid N$
of a primitive Hecke-Maass cusp form $\phi$ for $\Gamma _{0}(N)$ with Dirichlet 
character $\chif$. Then the Satake parameters at $p$ for 
$L(s, \adjoint \phi)$ are given by $\{\alpha _{p}/\beta _{p},\; 1, \; 
\beta _{p}/\alpha_{p}\}$. 
For any unramified $p\nmid N$, we have 
$$|\lambdaf(p)|^{2} = \lambda ^{2}_{\phi}(p) \overline{\chif}(p)
= \lambdaf(p^{2}) \overline{\chif}(p)
+ 1 \;.$$
In particular $\lambdaf(p^{2}) \overline{\chif}(p)$ is real
and $\lambdaf(p^{2}) \overline{\chif}(p) \geq -1$. 
If $p$ is not a Ramanujan prime of $\phi$, i.e., $|\alpha_p/\beta_p|\neq 1$,
then we have 
$|\lambdaf(p)| > 2$
and
$\alpha_p/\beta_p$ is real and $>0$ and for $n \geq 0$ 
$$\lambdaf(p^{2n})\overline{\chif}(p^{n}) 
\; = \; \frac{ \left( \sqrt{\frac{\alpha _{p}}{\beta _{p}}} \right)^{2n+1}
- \left( \sqrt{\frac{\beta _{p}}{\alpha _{p}}} \right)^{2n+1}}{
\sqrt{\frac{\alpha _{p}}{\beta _{p}}} 
-  \sqrt{\frac{\beta _{p}}{\alpha _{p}}}} > d(p^{2n})=2n+1,$$
where $d$ is the divisor function.
\end{lemma}

\begin{proof}
The first assertion follows from the definition of 
$L(s, \adjoint\phi)$ and the fact that the Satake parameters at $p$ 
for the contragredient form $\overline{\phi}$ are $\{\alpha _{p}^{-1}, \;
\beta _{p}^{-1}\}$. 
For $p\nmid N$, we have $$\lambdaf(p)=\chif(p)\overline{\lambdaf(p)}.$$ 
By Hecke relation, we have $\lambdaf(p^2)=\lambdaf(p)^2-\chif(p).$
Then we have $\lambdaf(p^2)\overline{\chif}(p)=\lambdaf(p)^2\overline{\chif}(p)-1=\lambdaf(p)\overline{\lambdaf(p)}-1$  and obviously
$\lambdaf(p^{2}) \overline{\chif}(p)$ is real and $\geq -1$. 

For $p\nmid N$, we have 
$$\alpha_p+\beta_p=\lambdaf(p)\quad \text{ and }\quad \alpha_p\beta_p=\chif(p).$$
Then we get 
$$\frac{\alpha_{p}}{\beta _{p}} + \frac{\beta _{p}}{\alpha _{p}}=|\lambdaf(p)|^2-2\geq -2\quad \text{ and }\quad \frac{\alpha_{p}}{\beta _{p}} \cdot \frac{\beta _{p}}{\alpha _{p}}=1.$$
The pair $\{\alpha_p/\beta_p,\beta_p/\alpha_p\}$ are the roots of the quadratic equation $$X^2-(|\lambdaf(p)|^2-2)X+1=0.$$ 
If $p\nmid N$ is not a Ramanujan prime of $\phi$, i.e., $|\alpha_p/\beta_p|\neq 1$,
this implies that $\{\alpha_p/\beta_p,\beta_p/\alpha_p\}$ are two real positive distinct roots. Because their product is $1$, one of them is $>1$ and the other is $<1$. Also, we have $|\lambdaf(p)|>2$.
From
$$\lambdaf(p^{n}) = \frac{\alpha _{p}^{n+1} -
\beta _{p}^{n+1}}{\alpha _{p} - \beta _{p}} \quad\text{ and }\quad \alpha_p\beta_p=\chif(p),$$
we get the last assertion.
\end{proof}

\section{Hecke-Maass cusp forms of level 1}\label{section2}
In this section, to illustrate quickly and clearly the main ideas of this 
paper,
we consider the simplest case of level $1$. 
Thus $\phi$ is a Hecke-Maass cusp form for $\ssl{2}{Z}$, 
with the Laplacian eigenvalue $\lambdaf =               
1/4 + t_{\phi}^{2}\;$ and
the $n$-th Hecke eigenvalue $\lambdaf(n)$.
The goal of this section is to prove the following theorem.

\begin{theorem}\label{thm1}
Let $\phi$ be a Hecke-Maass cusp form for $\ssl{2}{Z}$ as above. Then for any $\epsilon > 0$, there exists a prime $p$ such that $|\lambdaf(p)| \leq 2$ and $p \ll t_{\phi}^{8/11 + \epsilon}$, 
where the implied constant depends on $\epsilon>0$ alone.
\end{theorem}

\begin{remark} It is clear from the proof that the same argument
is in fact still valid for any primitive Hecke-Maass cusp form $\phi$ on $\gammao(N)$ with
the central character $\chif$, by simply replacing $\lambdaf(p^{2})$
by $\lambdaf(p^{2}) \overline{\chif}(p)$.
\end{remark}

\begin{proof}
Assume $p$ is not a Ramanujan prime of $\phi$ for all
primes $p\leq y$. Then by the Lemma \ref{lemma1} we have 
$\lambdaf(d^{2})>3$  for $1<d\leq y$. Take $x = yz$ and 
$z = y^{\delta}$ with $0 < \delta < 1/2$.
Consider the sum 
$$ S(x) = \sum _{d < x} \lambdaf(d^{2}) 
 \log \frac{x}{d} =  S^{+}(x) + S^{-}(x),$$ where $S^{+}(x)$ and $S^{-}(x)$ denote the partial sums over the positive
and negative eigenvalues $\lambdaf(d^{2})$ respectively.

If $\lambdaf(d^{2}) < 0$ in $S^{-}(x)$, then $d = m p$ 
with $\lambdaf(m^{2}) > 0$, $\lambdaf(p^{2}) < 0$,
where
all the prime divisors of $m$ do not exceed $y$, and $p > y$. 
From $\lambdaf(p^{2}) = \lambdaf^{2}(p) - 1 \geq  -1$,
we deduce that
\begin{eqnarray} S^{-}(x) & = & \sum_{\substack{pm < x, \; p > y, \;
\\\lambdaf(p^{2}) < 0}}
 \lambdaf((pm)^{2}) 
\log \bra{\frac{x}{pm}} \nonumber \\ & \geq & - \sum _{m < z}  
\lambdaf(m^{2})
  \sum _{p\leq x/m}  \log \bra{\frac{x}{pm}} \nonumber  \\\label{sminus}
& \geq & -  \left( \sum _{m < z} 
\frac{\lambdaf(m^{2})}{m}  \right) \frac{x}{\log y}
\left( 1 + O\left( \frac{1}{\log y}\right) \right),
\end{eqnarray}
in view of the asymptotics
$$\pi (x) \log x - \sum_{p\leq x}\log p=\frac{x}{\log x} + O\left(
\frac{x}{\log ^{2}x}\right),$$
by the Prime Number Theorem (see \cite{Pra}).

Next we bound $S^{+}(x)$. By positivity,
\begin{eqnarray} S^{+}(x) & \geq & \sum _{m < z} 
\lambdaf(m^{2})
  \sum_{\substack{l < x/m\\ p|l \Rightarrow z < p \leq y} }
\lambdaf(l^{2})  \log \bra{\frac{x}{lm}}  \nonumber\\\label{lower}
& \geq & 3 \sum_{m < z} \lambdaf(m^{2}) 
\Phi'(x/m, y, z), 
\end{eqnarray}
where 
\[ \Phi' (X, Y, Z) = \sum _{\substack{1 < l < X\\p|l \Rightarrow Z< p \leq Y} }
\log  \left( \frac{X}{l} \right).\]

\begin{lemma}\label{sieveb}
If $Z$ is large, $Z<Y$ and $Y <X \leq YZ$, then
$$\Phi' (X, Y, Z) > \frac{X}{2 \log Z} - \frac{X}{\log Y}+O\bra{\frac{Z\log Y}{\log Z}+\frac{X}{\log^2 Z}}.$$
\end{lemma}

\begin{proof}
Define          
$$\Phi(X,Y,Z)=\sum_{\substack{1<l<X\\p|l\Rightarrow Z<p\leq Y}}1 \quad\text{ and }\quad
\Phi(X,Z)=\sum_{\substack{1<l<X\\p|l\Rightarrow Z<p}}1.$$
Then we have $$\Phi'(X,Y,Z)=\int_{Y}^X\Phi(t,Y,Z)\frac{{\rm d} t}{t}+\int_Z^Y\Phi(t,Z)\frac{{\rm d} t}{t}.$$
For $Y <t \leq YZ$, it is easy to see that $$\Phi(t,Y,Z)=\Phi(t,Z)-\Phi(t,Y).$$
Recall the asymptotic formula of $\Phi(X,Z)$, $X\geq Z\geq 2$ (see Theorem 3, p. 400, \cite{Ten})
\begin{equation}\label{buchstab}\Phi(X,Z)=\omega\left(\frac{\log X}{\log Z}\right)\frac{X}{\log Z}
-\frac{Z}{\log Z}
+O\left(\frac{X}{\log^2 Z}\right),\end{equation}
where $\omega\bra{u}$ is the Buchstab  function, that is the continuous solution to the
difference-differential equation
$$u\omega(u) = 1 \quad(1\leq u \leq2),$$
$$\bra{u\omega(u)}' = \omega (u-1) \quad(u>2).$$
Moreover the range of the Buchstab function is  $1/2\leq \omega(u)\leq 1$.
We infer that
\begin{eqnarray*}
\Phi'(X,Y,Z)
&=&\int_Z^X\Phi(t,Z)\frac{{\rm d} t}{t}-\int_Y^X\Phi(t,Y)\frac{{\rm d} t}{t}\\
&\geq &\int_Z^X\bra{\frac{1}{2}\frac{t}{\log Z}
-\frac{Z}{\log Z}}\frac{{\rm d} t}{t}-\int_Y^X\bra{\frac{t}{\log Y}
-\frac{Y}{\log Y}}\frac{{\rm d} t}{t}+O\bra{\frac{X}{\log ^2Z}}\\
&\geq & \frac{X}{2 \log Z} - \frac{X}{\log Y}+O\bra{\frac{Z\log Y}{\log Z}+\frac{X}{\log^2 Z}}.
\end{eqnarray*}
This completes the proof of Lemma \ref{sieveb}.
\end{proof}

By Lemma \ref{sieveb}, we have 
$$\Phi'(x/m, y, z) > \left(\frac{1}{2\delta} - 
1 + O\left( \frac{1}{\log y}\right)\right)
\frac{x}{m \log y},$$
and 
$$S^{+}(x) >  \left(\frac{3}{2\delta} - {3}+
O\left( \frac{1}{\log y}\right)\right)
\left( \sum _{m < z}
\frac{\lambdaf(m^{2})}{m}  \right) \frac{x}{\log y}$$
from (\ref{lower}).
Consequently, after combining with the lower bound of $S^-(x)$ in (\ref{sminus}), we deduce that 
$$S(x) > \left(\frac{3}{2\delta} - 4+ O\left( \frac{1}{\log y}\right) \right)
\left( \sum _{m < z}
\frac{\lambdaf(m^{2})}{m}  \right) \frac{x}{\log y}.$$
Therefore 
\begin{equation}\label{eq2} S(x) \gg \frac{x}{\log x}, \end{equation}
on choosing $\delta = 3/8-\epsilon$, provided $y \gg 1$.

Now for $\sigma > 1$ and any $\epsilon > 0$, we have
\begin{eqnarray}\label{eq5} S(x) & = & \sum _{d < x} \lambdaf(d^{2})
 \log \bra{\frac{x}{d}} \nonumber \\ & = & \frac{1}{2\pi i} \int _{(\sigma)}
\frac{L(s,\adjoint\phi)}{\zeta (2s)} \frac{x^{s}}{s^{2}} \dd s 
\nonumber \\
& = & \frac{1}{2\pi i} \int_{(1/2)}
\frac{L(s,\adjoint \phi)}{\zeta (2s)} \frac{x^{s}}{s^{2}}\dd s 
\nonumber\\
& \ll & t_{\phi}^{1/2 + \eta} x ^{1/2},
\end{eqnarray}
by shifting the line of integration to $\Re(s) = 1/2$ and applying the 
convexity bound for $L(s,\adjoint\phi)$ on the critical line.

 Comparing (\ref{eq2}) and (\ref{eq5}), we obtain
$$x = y^{1 + \delta} \ll t_{\phi}^{1 + 2\eta},$$
i.e.
$$y \ll t_{\phi}^{8/11 + \epsilon},$$
for any $\epsilon > 0$. This completes the proof of Theorem \ref{thm1}.  \end{proof}

\begin{remark}
A hypothetical subconvexity bound of $L(s,\adjoint\phi)$ in the eigenvalue aspect on the critical line $\real(s)=1/2$ would yield
$$L(\frac{1}{2}+it,\adjoint \phi)\ll t_\phi^{1/2-\delta}t^{3/4+\epsilon},$$
for some $\delta>0$.
It is clear that this in turn would immediately lead to $y\ll t_\phi^{1-2\delta}$.
\end{remark}

\section{Refinement and Generalization}\label{section3}

In this section we refine the approach in Section \ref{section2} to obtain a better exponent. The method employs the theory of multiplicative functions.

Let $\phi$ be a primitive Hecke-Maass cusp form for $\gammao(N)\subset\ssl{2}{Z}$ 
with Dirichlet character $\chif:(\mathbb{Z}/N\mathbb{Z})^*\to \mathbb{C}$.
It has Laplacian eigenvalue $1/4+t_\phi^2$ with the parameter $t_\phi$ lying 
in $\mathbb{R}\cup [-7i/64,7i/64]$.
We assume that $\phi$ is not of dihedral type, otherwise the full Ramanujan conjecture is known.
The standard $L$-function of $\phi$ is given by $$L(s,\phi)=\summ\frac{\lambdaf (n)}{n^s},$$ where 
$\lambdaf (n)$'s are normalized Hecke eigenvalues with 
$\lambdaf (1)=1$ and $T_n \phi = \lambdaf (n)\phi$ for $n\in\mathbb{Z}$.

Our main tool is the adjoint  $L$-function of $\phi$ mentioned in the Introduction and Lemma \ref{lemma1}
$$L(s,\adjoint\phi)=\zeta^{(N)}(2s)\sum_{n=1}^\infty \frac{\lambdaf (n^2)\overline{\chif}(n)}{n^s}=\sum_{n=1}^\infty \frac{\gjf (n)}{n^s},$$
where $\gjf (n)=\sum_{k^2|n} \lambdaf (n^2/k^4)\overline{\chif(n/k^2)}$ for $(n,N)=1$.
As in \cite{IS}, we denote the analytic conductor by $$Q=Q(\adjoint\phi).$$ We have  
$$Q(\adjoint \phi)\leq N^2(1+|t_\phi|)^2.$$
 Lemma \ref{lemma1} implies that for a prime $p\nmid N$ then  $\gjf(p)$ is real and $\geq -1$.
It also implies $\gjf(p)>3$ if $p$ is not a Ramanujan prime of $\phi$, i.e., $|\lambdaf(p)|>2$.

Let us assume that $p$ is not a Ramanujan prime of $\phi$ for all $p\leq y$ and $p\nmid N$. Thus we have $\gjf (p)>3$ for all $p\leq y$. Define $$S^\flat(x)=\sumb{\substack{n\leq x\\(n,N)=1}}{} \gjf (n) \log\bra{\frac{x}{n}}$$ where the summation $\sumb{}{}$ is taken over squarefree numbers.

\begin{lemma}\label{convex}
We have 
$$S^\flat(x)\ll  x^{3/4}Q^{1/8+\epsilon}.$$
\end{lemma}

\begin{proof}
 Define $$G(s)=\prod_{p\nmid N}\bra{1-\frac{\gjf (p)}{p^s}+\frac{\gjf (p)}{p^{2s}}-\frac{1}{p^{3s}}}\bra{1+\frac{\gjf (p)}{p^s}}.$$
The analytic function $G(s)$ is absolutely convergent in $\{\real(s)>1/2+\epsilon\}$, and uniformly bounded by $Q^\epsilon$ with any $\epsilon>0$, in view 
of the Rankin-Selberg convolution of $\adjoint \phi \times \adjoint \phi$.
Now $$L(s,\adjoint \phi)G(s)=\sumb{\substack{n=1\\(n,N)=1}}{\infty}\; \frac{\gjf(n)}{n^s}$$ is absolutely convergent in $\{\real(s)>1\}$.
For $c>1$, 
\begin{eqnarray}
S^\flat(x)&=&\twopii\int_{(c)}L^{(N)}(s,\adjoint\phi)G(s)\frac{x^s}{s^2}\dd s\nonumber\\
\label{vertical}
&=&\twopii\int_{(3/4)}L^{(N)}(s,\adjoint\phi)G(s)\frac{x^s}{s^2}\dd s.
\end{eqnarray}
By using the convexity bound $$L(s,\adjoint \phi)
\ll_\epsilon (Qt^3)^{(1-\real(s))/2+\epsilon},$$ we obtain $S^\flat(x)\ll x^{3/4}Q^{1/8+\epsilon}$.
\end{proof}

Define a multiplicative function supported on squarefree numbers with 
$$h(p)=\begin{cases}3,&p\leq y,\\-1,&p>y.\end{cases}$$
It extends to all squarefree numbers. For convenience, we define $h(n)=0$ if $n$ is not squarefree. Define $$\sbb^\flat(x)=\sumb{\substack{n\leq x\\(n,N)=1}}{} \gjf (n).$$

\begin{lemma}\label{convolution1}
If $\sum\limits_{\substack{n\leq t\\(n,N)=1}}h(n)\geq 0$ for all $t\leq x$, we have
\begin{equation}\label{convolution2}
\sbb^\flat(x)\geq \sum_{\substack{n\leq x\\(n,N)=1}} h(n).
\end{equation}
\end{lemma}

\begin{proof}
The proof follows \cite{K}.
Let us define a multiplicative function $g$ defined by the Dirichlet convolution
$$\gjf=h\ast g, \quad\text{ or }\quad \gjf(n)=\sum_{d|n}h(d)g\bra{\frac{n}{d}}.$$
We have $g(p)=\gjf(p)-h(p)\geq 0$ for $p\nmid N$.
Then we have 
\begin{eqnarray*}
\sbb^\flat(x)&=&\sumb{\substack{n\leq x\\(n,N)=1}}{} \gjf (n)\\
&=&\sumb{\substack{n\leq x\\(n,N)=1}}{} \sum_{d|n}h(d)g\bra{\frac{n}{d}}\\
&=&\sumb{\substack{d\leq x\\(d,N)=1}}{} g(d)\sum_{\substack{b\leq x/d\\(b,N)=1}} h(b)\\
&\geq& \sum_{\substack{n\leq x\\(n,N)=1}} h(n)
\end{eqnarray*}
Both $g(d)$ and $\sum h(b)$ are non-negative. We have $g(1)=1$ and hence this lemma is proved.
\end{proof}

\begin{lemma}\label{convolution}
If $\sum\limits_{\substack{n\leq t\\(n,N)=1}}h(n)\geq 0$ for all $t\leq x$, 
we have $$S^\flat(x)\geq \sum_{\substack{n\leq x\\(n,N)=1}} h(n)\log\bra{\frac{x}{n}}.$$
\end{lemma}

\begin{proof}
It follows from the formula
$$S^\flat(x)=\int_1^x \sbb^\flat(t)\frac{\dd t}{t}$$ and Lemma \ref{convolution1}.
\end{proof}

The following lemma evaluates the mean of the multiplicative function $h(n)$ 
over a long range $1\leq n\leq x$ where $x$ equals $y^u$ for some $u>1$.
The special case of this lemma appears in \cite{K} and a more elaborate version is available in \cite{mato}.

\begin{lemma}\label{sieve1}
Let $U\geq 1$ and let $h(n)$ be as above. We have
$$\sum_{\substack{n\leq y^u\\(n,N)=1}}h(n)
=c(N)(\sigma(u)+o_{U}(1))(\log y)^{2}y^u$$
uniformly for $u\in [{1}/{U},U]$,
where $\lim\limits_{y\to\infty}o_{U}(1)=0$
and 
$c(N)={\left(\frac{\phi(N)}{N}\right)}^3\prod_{p\nmid N}(1-\frac{1}{p})^3(1+\frac{3}{p})\gg (\log\log N)^{-3}$.
The constant $\sigma(u)$ is the continuous function of $u\in (0,\infty)$ 
uniquely determined by the differential-difference equation
\begin{eqnarray}
\sigma(u)&=&u^2,\;\quad \;\quad\;\quad\;\quad 0<u\leq 1,\nonumber\\
(u^{-2}\sigma(u))'&=&-\frac{4\sigma(u-1)}{u^3},\quad u>1.\nonumber
\end{eqnarray}
\end{lemma}

\begin{proof}
In Lemma 6 of \cite{mato}, take $K=1$, $x_0=0$, $x_1=1$,
$\chi_0=3$, $\chi_1=-1$, $q=1$.
The function $\sigma(u)$ can be computed from Lemma 8 of \cite{mato}. 
\end{proof}

\begin{lemma}\label{sieve}
Let $u_0>1$ be such that $\sigma(u)>0$ for $1<u\leq u_0$.
We have for $y\gg_{u_0} 1$, 
$$\sum_{\substack{n\leq y^{u_0}\\(n,N)=1}} h(n)\log\bra{\frac{y^{u_0}}{n}}\gg_{u_0} c(N) y^{u_0}.$$
\end{lemma}

\begin{proof}
Define $$H(x)=\sum_{\substack{n\leq x\\(n,N)=1}} h(n).$$
We have 
\begin{eqnarray*}
\sum_{\substack{n\leq y^{u_0}\\(n,N)=1}} h(n)\log\bra{\frac{y^{u_0}}{n}}&=&\int_1^{y^{u_0}} H(t)\frac{\dd t}{t}=\int_0^{{u_0}} H(y^{u}) \log y \dd u\\
&\geq&\int_{1/{u_0}}^{u_0} H(y^{u}) \log y \dd u\\
\end{eqnarray*}
By Lemma \ref{sieve1},
we have  for $1/u_0 \leq u \leq u_0$ uniformly
$$H(y^u)=c(N)(\sigma(u)+o_{u_0}(1))(\log y)^{2}y^u.$$ 
For $y\gg_{u_0} 1$, we hence have 
$$\int_{1/{u_0}}^{u_0} H(y^{u}) \log y \dd u
\gg_{u_0} c(N) y^{u_0}$$
and this completes the proof.
\end{proof}

Let $u_0$ be the same as defined in Lemma \ref{sieve}. We have $c(N)\gg Q^{-\epsilon}$ for $\epsilon>0$.
Comparing Lemma \ref{convex}, Lemma \ref{convolution} and Lemma \ref{sieve},
we infer that  
$$y^{u_0}Q^{-\epsilon}\ll_{u_0}\sum_{\substack{n\leq y^{u_0}\\(n,N)=1}} h(n)\log\bra{\frac{y^{u_0}}{n}}\ll S^\flat(y^{u_0})\ll \bra{y^{u_0}}^{3/4}Q^{1/8+\epsilon}$$
and this in turn gives 
\begin{equation}\label{twominussix}
y\ll_{u_0} Q^{\frac{1}{2u_0} + \epsilon}.
\end{equation}

By numerical computation of \textit{Mathematica}, we find  the smallest zero  of $\sigma(u)$ is approximately
$3.65887.$ Then taking $u_0$ to be microscopically less than $3.65887$ we get:

\begin{theorem}\label{mainx}
For any primitive Hecke-Maass cusp form $\phi$ for $\gammao(N)$ with character $\chif$ and Laplace eigenvalue $1/4+t_\phi^2$, there exists a prime number $p\nmid N$ with $p\ll(N(1+|t_{\phi}|))^{0.27332}$ such that the Ramanujan conjecture holds for $\phi$ at $p$.
\end{theorem}

\begin{remark}
In Lemma \ref{convex}, the line of integration in  (\ref{vertical}) may be taken on $\{\real\bra{s}=\sigma\}$ instead of  $\{\real\bra{s}=\sigmaa\}$ for $1/2<\sigma<1$. This will result in a different version of Lemma \ref{convex}, i.e.,
$$S^\flat(x)\ll x^{\sigma}Q^{(1-\sigma)/2+\epsilon}.$$
However, this change has no impact on the final exponent in Theorem \ref{mainx}. 
\end{remark}

\begin{remark}
To estimate the smallest zero of $\sigma(u)$ without numerical computation, we have from Lemma \ref{sieve1}
$$\sigma(u)=7u^2-8u+2-4u^2\log u$$
for $1\leq u\leq 2$. It is not hard to prove that $\sigma(u)$ is monotone  for $1\leq u\leq 2$ and this leads us to conclude $\sigma(u)$ is positive for $1\leq u\leq 2$. Without numerical computation, we can have 1/4 as the exponent in Theorem \ref{mainx}.

For $2\leq u\leq 3$, we have 
\begin{multline*}
\sigma(u)=
16 u^2 Li_2(1-u)+(4 \pi^2 u^2)/3+35 u^2-24 u^2 \log(u-1)+16 u^2 \log(u-1) \log(u)\\-4 u^2 \log(u)-80 u+32 u \log(u-1)-8 \log(u-1)+34,
\end{multline*}
where $Li_2$ is the famous dilogarithm function (see \cite{Zag}). We leave to 
the reader to verify that $\sigma(u)$ is positive for $2\leq u\leq 3$.
\end{remark}

\section{Natural Density of Ramanujan Primes}\label{four}

Let $\phi$ be a primitve Maass form for $\gammao(N)$ with character $\chif$ and with Hecke eigenvalues $\lambdaf(n)$, following the same notations of the previous sections. 
We assume that $\phi$ is not of Artin type, 
since otherwise the full Ramanujan conjecture is known (\cite{KSh}).

In this section, we refine the density results of the Ramanujan primes in \cite{rama} and \cite{KSh} from \textit{Dirichlet density} to \textit{natural density}. We achieve the same constant by employing a similar but different method.
We will first quickly indicate how our method leads directly to 
the fact that the lower
{\it natural density} of the Ramanujan primes of $\phi$ is at least $9/10$,
and then improve it further to $34/35$ by a more elaborate argument.

The adjoint (Gelbart-Jacquet) lift  (see \cite{GJ}) of  $\phi$,
with its $L$-function defined by $L(s, \adjoint\phi)=\summ {\gjf(n)}/{n^s}, \Re(s) > 1$ where $\gjf(p)=\lambdaf(p^2)\overline{\chif}(p)$, is a cuspidal automorphic 
representation of $GL(3)$. The symmetric cube lift $\sym^3\phi$ and the twisted symmetric fourth power lift $\sym^4\phi\times\overline{\chif}^2$ are cuspidal automorphic representations of $GL(4)$ and $GL(5)$ respectively (see \cite{KSh2} and \cite{Ki}).
Let $$L(s,\sym^3\phi)=\summ \frac{\gjft(n)}{n^s}\andd L(s,\sym^4\phi\times \overline{\chif}^2)=\summ \frac{\gjff(n)}{n^s}$$ 
be their $L$-functions and $\{\alpha_p,\beta_p\}$ be the Satake parameters associated with $\phi$ at an unramified prime $p$. 
The Satake parameters of $\sym^3\phi$ are given by $\{\alpha_p^3, \alpha_p^2\beta_p,\alpha_p\beta_p^2,\beta^3_p\}$,
while those of $\sym^4\phi\times \overline{\chif}^2$ are given by 
$\{\alpha_p^2/\beta_p^2,\alpha_p/\beta_p, 1, \beta_p/\alpha_p,\beta_p^2/\alpha_p^2\}$.


In light of  the standard zero-free region of $L(s,\adjoint\phi)$ and $L(s,\sym^4\phi\times\overline{\chif}^2)$, the following Prime Number Theorem for $L$-functions holds. (see Theorem 5.13 of \cite{IK}).
\begin{lemma}\label{PNTlemma1}
We have 
\begin{equation*}
\sum_{p\leq X} \gjf(p) = o\bra{\frac{X}{\log X}}
\andd
\sum_{p\leq X} \gjff(p) = o\bra{\frac{X}{\log X}},
\end{equation*}
as $X\to \infty$.
\end{lemma}

For the result on the natural density, let us first consider the sum
$$ S(X) = \sum _{p \leq X} (1 + 3 A_{\phi}(p))^{2}.$$
On one hand, we have
\[ S(X) > 10^{2} \#\{p\leq X, |\lambdaf(p)| > 2\} .\]
On the other hand, we have
\begin{eqnarray*} S(X) & = & \sum _{p\leq X} (1 + 6 A_{\phi}(p) + 
9 A_{\phi}(p)^{2}) \\
& = & \sum _{p\leq X} (1 + 6 A_{\phi}(p) + 9 (A_{\phi}^{[4]}(p) + A_{\phi}(p) + 1) )\\
& = & \sum _{p\leq X} (10 + 15 A_{\phi}(p) +
9 A_{\phi}^{[4]}(p)) \\
& = & 10 \pi(X) + o(\pi(X)),
\end{eqnarray*}
by Lemma \ref{PNTlemma1}. Hence we get
\[ \#\{p\leq X, |\lambda _{\phi}(p)| > 2\} \leq \frac{1}{10}  \pi(X) + 
o(\pi(X)),\]
or equivalently 
\[ \#\{p\leq X, |\lambda _{\phi}(p)| \leq 2\} \geq \frac{9}{10}  \pi(X) +
o(\pi(X)),\]
i.e., the lower
{\it natural density} of the Ramanujan primes of $\phi$ is at least $9/10$. 

Next we turn to the improvement of the above density result. A zero-free region of Rankin-Selberg $L$-functions has been established by Moreno in \cite{Mor}.
By the Tauberian theorem of Wiener and Ikehara (Theorem 1, page 311, \cite{La})
for $L'/L(s)$, where 
$L(s) = L(s, \Pi \times \overline{\Pi})$, and
$\Pi = \sym^3\phi$ or $\sym^4\phi$, we obtain the Prime Number Theorem for $L(s)$. 
\begin{lemma}
Let $\Lambda$ be the  von Mangoldt function.
We have 
$$\sum_{n\leq X} \Lambda(n)|\gjft(n)|^2 \sim X\andd \sum_{n\leq X} \Lambda(n)|\gjff(n)|^2 \sim X,$$
as $X\to \infty$.
\end{lemma}

\begin{remark}\label{weakpnt}
The previous lemma implies $$\limsup_{X\to\infty}\frac{\sum_{p\leq X} |\gjff(p)|^2}{\pi(X)}\leq 1\andd \limsup_{X\to\infty}\frac{\sum_{p\leq X} |\gjft(p)|^2}{\pi(X)}\leq 1.$$
\end{remark}

\begin{theorem}\label{natural}
We have
$$\liminf_{X\to\infty} \frac{\#\{p\leq X,  |\lambda _{\phi}(p)| \leq 2\}}{\pi(X)}
\geq \frac{34}{35}.$$
\end{theorem}

\begin{proof}
Define $U(p)=\left(1+3\gjf(p)+5\gjff(p)\right)^2$.
Obviously we have $U(p)\geq 0$ and if $p$ is not a Ramanujan prime, we have $$U(p)>35^2.$$
By the Hecke relations
$$\gjf(p) \gjff(p)=|\gjft(p)|^2-1\andd \gjf(p)^2=\gjff(p)+\gjf(p)+1,$$
we have 
\begin{eqnarray}
U(p)&=&1+9\gjf(p)^2+25\gjff(p)^2+6\gjf(p)+10\gjff(p)+30\gjf(p) \gjff(p)\nonumber\\
&=&-20+15\gjf(p)+19\gjff(p)+30|\gjft(p)|^2+25\gjff(p)^2. \nonumber
\end{eqnarray}
By the previous two lemmas, we have 
\begin{equation}\label{sup}
\limsup_{X\to\infty}\frac{\sum_{p\leq X}U(p)}{\pi(X)}\leq 35.
\end{equation}
We have 
\begin{equation*}
\frac{\sum_{p\leq X}U(p)}{\pi(X)}\geq \frac{35^2(\pi(X)-\raman)}{\pi(X)}
\end{equation*}
and then $$\frac{\raman}{\pi(X)}\geq 1-\frac{\sum_{p\leq X}U(p)}{35^2\pi(X)}.$$ Hence by (\ref{sup}) we get $\liminf\limits_{X\to\infty}{\raman}/{\pi(X)}\geq 34/35$.
\end{proof}


$$\;$$
\noindent {\scshape{Wenzhi Luo}} and {\scshape{Fan Zhou}} \\
Department of Mathematics\\
The Ohio State University\\
Columbus, OH 43210, USA\\
wluo@math.ohio-state.edu, zhou.1406@math.osu.edu

\begin{thebibliography}{1}

\bibitem[De1]{De1} Deligne, Pierre. "Formes modulaires et 
repr\'esentations $l$-adiques." In \textit{S\'eminaire Bourbaki vol. 1968/69 Expos\'es 347-363}, pp. 139-172. Springer Berlin Heidelberg, 1971.





\bibitem[De2]{De2} Deligne, Pierre. "La conjecture de Weil. I." \textit{Publications Math\'ematiques de l'Institut des Hautes \'Etudes Scientifiques} 43, no. 1 (1974): 273-307.



\bibitem[GJ]{GJ}Gelbart, Stephen, and Herv\'e Jacquet. "A relation between automorphic representations of GL(2) and GL(3)." \textit{Ann. Sci. \'Ecole Norm. Sup.(4)} 11, no. 4 (1978): 471-542.



\bibitem[IKS]{IKS} Iwaniec, Henryk, Winfried Kohnen, and J. Sengupta. "The first negative Hecke eigenvalue." \textit{International Journal of Number Theory} 3, no. 03 (2007): 355-363.





\bibitem[IK]{IK} Iwaniec, Henryk, and Emmanuel Kowalski. \textit{Analytic number theory}. Vol. 53. Providence: American Mathematical Society, 2004.






\bibitem[IS]{IS} Iwaniec, Henryk, and Peter Sarnak. "Perspectives on the analytic theory of L-functions." In \textit{Visions in Mathematics}, pp. 705-741. Birkh\"auser Basel, 2010.



\bibitem[Ki]{Ki}Kim, Henry. "Functoriality for the exterior square of $GL_4$₄ and the symmetric fourth of $GL_2$₂." \textit{Journal of the American Mathematical Society} 16, no. 1 (2003): 139-183.


\bibitem[KS]{KS} 
Kim, Henry, and Peter Sarnak. "Refined estimates towards the Ramanujan and Selberg conjectures." \textit{Journal of the American Mathematical Society} 16, no. 1 (2003): 175-181.



\bibitem[KSh]{KSh}
Kim, Henry H., and Freydoon Shahidi. "Cuspidality of symmetric powers with applications." \textit{Duke Mathematical Journal} 112, no. 1 (2002): 177-197.



\bibitem[KSh2]{KSh2}
Kim, Henry H., and Freydoon Shahidi. "Functorial products for $GL_2\times GL_3$ and the symmetric cube for $GL_2$." \textit{Annals of Mathematics} (2002): 837-893.




\bibitem[KLSW]{K}Kowalski, Emmanuel, Yuk Kam Lau, Kannan Soundararajan, and Jie Wu. "On modular signs." In \textit{Mathematical Proceedings of the Cambridge Philosophical Society}, vol. 149, no. 3, pp. 389-411. Cambridge University Press, 2010.

\bibitem[Lan]{La}Lang, Serge. \textit{Algebraic number theory}. 
second edition, Graduate Texts in Mathematics, 110,
Springer-Verlag, New York, 1994.






\bibitem[Mat]{mato}Matom\"aki, Kaisa. "On signs of Fourier coefficients of cusp forms." In \textit{Mathematical Proceedings of the Cambridge Philosophical Society}, vol. 152, no. 02, pp. 207-222. Cambridge University Press, 2012.


\bibitem[Mor]{Mor}Moreno, Carlos J. "Analytic proof of the strong multiplicity one theorem." \textit{American Journal of Mathematics} (1985): 163-206.


\bibitem[Pra]{Pra}Prachar, Karl. \textit{Primzahlverteilung.} Springer-Verlag, Berlin, 1957.




\bibitem[Ram]{rama}Ramakrishnan, Dinakar. "On the coefficients of cusp forms." \textit{Mathematical Research Letters} 4 (1997): 295-308.




\bibitem[Ten]{Ten}Tenenbaum, G\'erald.   \textit{Introduction to analytic and probabilistic number theory}. Vol. 46. Cambridge University Press, 1995.



\bibitem[Zag]{Zag}
Zagier, Don. "The dilogarithm function." In \textit{Frontiers in number theory, physics, and geometry II}, pp. 3-65. Springer-Verlag, Berlin-Heidelberg, 2007.
\end{thebibliography}
\end{document}